\documentclass
{amsart}
\usepackage{graphicx}
\newtheorem{thm}{Theorem}[section]
\theoremstyle{definition}

\newtheorem{prop}[thm]{Proposition}
\newtheorem{defn}[thm]{Definition}
\newtheorem{lem}[thm]{Lemma}

\newtheorem{rem}[thm]{Remark}
\newtheorem{ex}[thm]{Example}

\numberwithin{equation}{section}
\begin{document}
{\footnotetext{This research was supported with a grant from Farhangian University}
\title[$S$-2-absorbing submodules and $S$-2-absorbing second submodules]{$S$-2-absorbing submodules and $S$-2-absorbing second submodules}

\author{Faranak
 Farshadifar}
\address{Assistant Professor, Department of Mathematics, Farhangian University, Tehran, Iran.}
\email{f.farshadifar@cfu.ac.ir}

\begin{abstract}
Let $R$ be a commutative ring with identity, $S$ be a multiplicatively closed subset of $R$, and let $M$ be an $R$-module.
In this paper, we introduce the notion of $S$-2-absorbing second submodules of $M$ as a generalization of $S$-second submodules and strongly 2-absorbing second submodules of $M$.
Let $N$ be a submodule of $M$ such that $Ann_R(N) \cap S= \emptyset$. We say that $N$ is a \textit{$S$-2-absorbing second submodule of $M$} if  there exists a fixed $s \in S$ and whenever $abN\subseteq K$, where $a, b \in R$ and $K$ is a submodule of $M$, implies either that $saN\subseteq K$ or $sbN\subseteq K$ or $sabN=0$. We investigate some properties of this class of submodules. Also, we obtain some results concerning $S$-2-absorbing submodules of $M$.
\end{abstract}


\maketitle
\section{Introduction}
\noindent
Throughout this paper, $R$ will denote a commutative ring with
identity and $\Bbb Z$ will denote the ring of integers.

Consider a non-empty subset $S$ of $R$. We call $S$ a multiplicatively closed subset (briefly, m.c.s.) of $R$ if (i) $0 \not \in S$, (ii) $1 \in S$, and (iii) $s\acute{s} \in S$ for all $s, \acute{s} \in S$ \cite{WK16}. Note that $S = R-P$ is a m.c.s. of $R$ for every prime ideal $P$ of $R$.
Let $M$ be an $R$-module. A proper submodule $P$ of $M$ is said to be \emph{prime} if for any $r \in R$ and $m \in M$ with $rm \in P$, we have $m \in P$ or $r \in (P:_RM)$ \cite{Da78}. A non-zero submodule $N$ of $M$ is said to be \emph{second} if for each $a \in R$, the homomorphism $ N \stackrel {a} \rightarrow N$ is either surjective
or zero \cite{Y01}.

Let $S$ be a m.c.s. of $R$ and let $P$ be a submodule of an $R$-module $M$ with $(P :_R M) \cap S =\emptyset$. Then the submodule
$P$ is said to be an \emph{$S$-prime submodule} of $M$ if there exists a fixed $s\in S$, and whenever $am \in P$, then $sa \in (P :_R M)$ or
$sm \in P$ for each $a \in R$, $m \in M$ \cite{satk19}. Particularly, an ideal $I$ of $R$ is said to be an \emph{$S$-prime ideal} if $I$ is an $S$-prime submodule of the $R$-module $R$.
In \cite{FF22} F. Farshadifar,  introduced  the notion of $S$-second submodules as a dual notion of $S$-prime submodules and investigated some properties of this class of modules.
Let $S$ be a m.c.s. of $R$ and $N$ a submodule of an $R$-module $M$ with $Ann_R(N) \cap S =\emptyset$. Then the submodule
$N$ is said to be an \emph{$S$-second submodule} of $M$ if there exists a fixed $s\in S$, and whenever $rN\subseteq K$, where $r \in R$ and $K$ is a submodule of $M$, then  $rsN=0$ or $sN\subseteq K$ \cite{FF22}.

The notion of 2-absorbing ideals as a generalization of prime ideals was introduced and studied in \cite{Ba07}. A proper ideal $I$ of $R$ is called a \emph{2-absorbing ideal}
of $R$ if whenever $a, b, c \in R$ and $abc \in I$, then $ab \in I$ or
$ac \in I$ or $bc \in I$ \cite{Ba07}. The authors in \cite{YS11} and \cite{pb12}, extended 2-absorbing ideals
 to 2-absorbing submodules. A proper submodule $N$ of an $R$-module $M$ is called a \emph{2-absorbing submodule }of  $M$ if whenever $abm \in N$ for some $a, b \in R$ and $m \in M$, then $am \in N$ or $bm \in N$ or $ab \in (N :_R M)$. In \cite{AF16}, the authors introduced the dual notion of 2-absorbing submodules (that is, \emph{strongly 2-absorbing second submodules}) of $M$ and investigated some properties of these classes of modules.
A non-zero submodule $N$ of $M$ is said to be a \emph{strongly 2-absorbing second submodule of} $M$ if whenever  $a, b \in R$, $K$ is a submodule of $M$,
and $abN\subseteq K$, then $aN\subseteq K$ or $bN\subseteq K$ or $ab \in Ann_R(N)$ \cite{AF16}.

Let $M$ be an $R$-module and $S$ be a m.c.s. of $R$.
In \cite{uatk20}, the authors introduced the notion of $S$-2-absorbing submodules of $M$ which is a generalization of $S$-prime submodules and 2-absorbing submodules and investigated some properties of this class of submodules.
A submodule $P$ of $M$ is said to be an \textit{$S$-2-absorbing} if $(P :_R M)\cap S =\emptyset$
and there exists a fixed $s \in S$ such that $abm \in P$ for some $a, b \in R$ and
$m \in M$ implies that $sab \in (P :_R M)$ or $sam\in P$ or $sbm \in P$. In particular, an ideal $I$ of $R$ is said to be an \textit{$S$-2-absorbing ideal} if $I$ is an $S$-2-absorbing submodule of the $R$-module $R$ \cite{uatk20}.

Let $S$ be a m.c.s. of $R$ and $M$ be an $R$-module.
The main purpose of this paper is to introduce the notion of $S$-2-absorbing second submodules of $M$ as a generalization of $S$-second submodules and strongly 2-absorbing second submodules of $M$. We provide some information about this class of submodules.  Moreover, we investigate some properties of $S$-2-absorbing submodules of $M$.
\section{$S$-2-absorbing submodules}
The following theorem gives a useful characterization of $S$-2-absorbing submodules.
\begin{thm}\label{l181.4}
Let $S$ be a m.c.s. of $R$ and $N$ be a submodule of an $R$-module $M$ with $(N :_R M)\cap S =\emptyset$. Then $N$ is $S$-2-absorbing if and only if there is a fixed $s \in S$ such that for every $a, b \in R$, we have either $(N :_Ms^2ab) = (N: _Ms^2a)$ or $(N :_Ms^2ab) = (N: _Ms^2b)$ or $(N :_Ms^3ab) =M$.
\end{thm}
\begin{proof}
Let $N$ be an $S$-2-absorbing submodule of $M$ and $m \in  (N :_Ms^2ab)$. Then $(sa)(sb)m \in N$. So by assumption, either $s^2am \in N$ or $s^2bm \in N$ or $abs^3 \in (N :_R M)$. If $s^2am \in N$ or $s^2bm \in N$, then $(N :_Ms^2ab) \subseteq  (N: _Ms^2a) \cup (N: _Ms^2b)$. Clearly, $(N: _Ms^2a) \cup (N: _Ms^2b)\subseteq (N :_Ms^2ab)$. So, $(N: _Ms^2a) \cup (N: _Ms^2b)= (N :_Ms^2ab)$. As $N$ is a submodule of $M$, it cannot be written as union of two distinct submodules. Thus $(N :_Ms^2ab) = (N: _Ms^2a)$ or $(N :_Ms^2ab) = (N: _Ms^2b)$. If $abs^3 \in (N :_R M)$, then $(N :_Ms^3ab) =M$.
Conversely, let $a, b \in R$ and $m \in M$ such that $abm \in N$. Then $m \in (N :_R s^2ab)$. By given hypothesis,
we have $(N :_Ms^2ab) = (N: _Ms^2a)$ or $(N :_Ms^2ab) = (N: _Ms^2b)$ or $(N :_Ms^3ab) =M$. Thus $s^2am \in N$
or $s^2bm \in N$ or $s^3ab \in (N :_R M)$.  Hence, $s^3am \in N$
or $s^3bm \in N$ or $s^3ab \in (N :_R M)$. Now by setting $s_1=s^3$, we get the result.
\end{proof}

\begin{lem}\label{l181.5}
\cite[Lemma 3.2]{AF16} Let $N$ be a submodule of an $R$-module $M$ and $r \in R$.
Then for every flat $R$-module $F$, we have $F\otimes(N :_M r) = (F \otimes N : _Mr)$.
\end{lem}

\begin{thm}\label{t181.6}
Let $S$ be a m.c.s. of $R$, $N$ be an $S$-2-absorbing submodule of an $R$-module $M$, and $F$ be a flat $R$-module. If
$Ann_R(F \otimes N) \cap S=\emptyset$, then $F \otimes N$ is an $S$-2-absorbing submodule of $F \otimes M$.
\end{thm}
\begin{proof}
Since $N$ is an $S$-2-absorbing submodule of $M$, by Theorem \ref{l181.4}, we have either $(N :_Ms^2ab) = (N: _Ms^2a)$ or $(N :_Ms^2ab) = (N: _Ms^2b)$ or $(N :_Ms^3ab) =M$ for $a,b \in R$. Assume that $(N :_Ms^2ab) = (N: _Ms^2a)$. Then by Lemma \ref{l181.5},
we have
$$
(F \otimes N : _Ms^2ab) = F \otimes (N :_M s^2ab) = F \otimes (N :_Ms^2a) = (F \otimes N :_M s^2a).
$$
If $(N :_Ms^3ab) = M$, then by Lemma \ref{l181.5}, we have
$$
(F \otimes N :_Ms^3ab) = F \otimes (N :_Ms^3ab) =F \otimes M.
$$
Hence by Theorem \ref{l181.4}, $F \otimes N$ is $S$-2-absorbing submodule of $F \otimes M$.
\end{proof}

\begin{thm}\label{t181.7}
Let $S$ be a m.c.s. of $R$ and $F$ be a faithfully flat $R$−module. Then $N$ is an $S$-2-absorbing submodule of $M$ if
and only if $F \otimes N$ is an $S$-2-absorbing submodule of $F \otimes M$.
\end{thm}
\begin{proof}
Let $N$ be an $S$-2-absorbing submodule of $M$. Suppose $Ann_R(F \otimes N) \cap S \not =\emptyset$. Then there is an $t \in Ann_R(F \otimes N) \cap S$. Thus  $F \otimes tN=0$. Hence,  $0 \rightarrow F \otimes tN \rightarrow 0 $ is an exact sequence. Since $F$ is a faithfully flat, $0 \rightarrow tN \rightarrow 0 $ is an exact which implies that $tN = 0$. Thus  $Ann_R(N) \cap S \not =\emptyset$, this is a contradiction. So $Ann_R(F \otimes N) \cap S =\emptyset$. Now by Theorem \ref{t181.6}, we have $F \otimes N$ is an $S$-2-absorbing submodule of $F \otimes M$.
Conversely, suppose $F \otimes N$ is an $S$-2-absorbing submodule of $F \otimes M$. Then $Ann_R(F \otimes N) \cap S= \emptyset$
implies that $Ann_R(N) \cap S = \emptyset$. Let $a, b \in R$. Then by Theorem \ref{l181.4}, we can assume that $(F \otimes N :_Ms^2ab) = (F \otimes N :_Ms^2a)$. By Lemma \ref{l181.5}, we have
$$
F \otimes (N :_Ms^2ab) = (F \otimes N :_Ms^2ab) = (F \otimes N :_Ms^2a) = F \otimes (N :_Ms^2a).
$$
So, $0 \rightarrow
F \otimes (N : _Ms^2ab) \rightarrow F \otimes (N :_Ms^2a) \rightarrow 0$ is an exact sequence. As $F$ is a faithfully flat, $0 \rightarrow (N : _Ms^2ab) \rightarrow (N :_Ms^2a) \rightarrow 0$ is an exact sequence. Thus $(N :_Ms^2ab) = (N : _Ms^2a)$ and so by Theorem \ref{l181.4}, $N$ is $S$-2-absorbing. If $(F \otimes N :_Ms^3ab) = F \otimes M$, then $F \otimes (N :_Ms^3ab) = (F \otimes N :_Ms^3ab) = F \otimes M$. So,
$$
0 \rightarrow F \otimes (N : _Mabs^3) \rightarrow F \otimes M \rightarrow 0
$$
is an exact sequence. As $F$ is a faithfully flat, $0 \rightarrow (N :_Ms^3ab) \rightarrow M \rightarrow 0$ is an exact sequence. Thus $(N :_Ms^3an) = M$.
Hence $N$ is an $S$-2-absorbing submodule of $M$.
\end{proof}

\begin{prop}\label{t191.15}
Let $S$ be a m.c.s. of $R$ and $N$ be an $S$-2-absorbing submodule of an $R$-module $M$. Then the following statements hold for some $s \in S$.
\begin{itemize}
\item [(a)] $(N:_Mth)\subseteq  (N:_Mts)$ or $(N:_Mth)\subseteq  (N:_Msh)$ for all $t,h \in S$.
\item [(b)] $((N:_RM):_Rth)\subseteq  ((N:_RM):_Rts)$ or $((N:_RM):_Rth)\subseteq  ((N:_RM):_Rsh)$ for all $t,h \in S$.
\end{itemize}
\end{prop}
\begin{proof}
(a) Let $N$ be an $S$-2-absorbing submodule of $M$. Then there is a fixed $s \in S$. Take an element $m\in (N:_Mth)$, where $t,h \in S$. Then $stm \in N$ or $shm \in N$ or $sth \in (N:_RM)$. As $(N:_RM)\cap S=\emptyset$, we have  $sth \not\in (N:_RM)$.
If for each $m\in (N:_Mth)$, we have $stm \in N$ (resp. $shm \in N$), then we are done. So suppose that there are $m_1\in (N:_Mth)$ such that $stm_1 \not\in N$ and $m_2\in (N:_Mth)$ such that $shm_2\not \in N$. Then we conclude that  $shm_1 \in N$ and $stm_2 \in N$.
Now $ht(m_1+m_2)\in N$ implies that $hs(m_1+m_2)\in N$ or $st(m_1+m_2)\in N$. Thus $stm_1 \in N$ or $shm_2\in N$, which is a desired contradiction.

(b) This follows from part (a).
\end{proof}

\begin{lem}\label{p9.11}
Let $S$ be a m.c.s. of $R$ and $I$ be an $S$-2-absorbing ideal of $R$. Then $\sqrt{I}$
is an $S$-2-absorbing ideal of $R$ and there is a fixed $s \in S$ such that $sa^2 \in I$ for every $a\in\sqrt{I}$.
\end{lem}
\begin{proof}
Clearly, as $I$ is an $S$-2-absorbing ideal of $R$, there is a fixed $s \in S$ such that $sa^2 \in I$ for every $a\in\sqrt{I}$. Now let
$a,b,c \in R$ such that $abc\in\sqrt{I}$. Then  $sa^2b^2c^2=s(abc)^2\in I$. Since $I$ is a $S$-2-absorbing ideal of $R$, we may assume that $s^2a^2b^2\in I$. This implies that  $sab\in\sqrt{I}$, as needed.
\end{proof}

Recall that an $R$-module $M$ is said to be a
\emph{multiplication module} if for every submodule $N$ of $M$
there exists an ideal $I$ of $R$ such that $N=IM$ \cite{Ba81}.

Let $N$ be a proper submodule of an $R$-module $M$. Then the \textit{$M$-radical} of
$N$, denoted by $rad(N)$, is defined to be the intersection of all prime submodules of
$M$ containing $N$ \cite{MM86}.
\begin{thm}\label{t181.3}
Let $S$ be a m.c.s. of $R$ and $M$ a finitely generated multiplication $R$-module.
If $N$ is an $S$-2-absorbing submodule of $M$, then $rad (N)$ is an $S$-2-absorbing submodule
of $M$.
\end{thm}
\begin{proof}
Since $N$ is an $S$-2-absorbing submodule of $M$, we have $(N:_R M)$ is a
$S$-2-absorbing ideal of $R$ by \cite[Proposition 3]{uatk20}. Thus by Lemma \ref{p9.11}, $\sqrt{(N :_R M)}$ is an $S$-2-absorbing ideal of $R$. By \cite[Theorem 4]{MM86}, $(rad(N):_R M))=\sqrt{(N :_R M)}$.  Therefore, $(rad (N):_R M)$ is an $S$-2-absorbing ideal of $R$. Now the result follows from \cite[Proposition 3]{uatk20}.
\end{proof}
\section{$S$-2-absorbing second submodules}
\begin{defn}\label{d2.1}
Let $S$ be a m.c.s. of $R$ and $N$ be a submodule of an $R$-module $M$ such that $Ann_R(N) \cap S= \emptyset$.
We say that $N$ is a \textit{$S$-2-absorbing second submodule of $M$} if  there exists a fixed $s \in S$ and whenever $abN\subseteq K$, where $a, b \in R$ and $K$ is a submodule of $M$, implies either that $saN\subseteq K$ or $sbN\subseteq K$ or $sabN=0$. In particular, an ideal $I$ of $R$ is said to be an \textit{$S$-2-absorbing second ideal} if $I$ is an $S$-2-absorbing second submodule of the $R$-module $R$. By a \textit{$S$-2-absorbing second module}, we mean
a module which is a $S$-2-absorbing second submodule of itself.
\end{defn}

\begin{lem}\label{l1.3}
Let $S$ be a m.c.s. of $R$, $I$ an ideal of $R$, and let $N$ be an $S$-2-absorbing second submodule of $M$. Then there exists a fixed $s \in S$ and whenever
 $a\in R$, $K$ is a submodule of $M$, and $IaN \subseteq K$, then  $asN \subseteq K$ or $IsN \subseteq K$ or $Ias \in Ann_R(N)$.
\end{lem}
\begin{proof}
Let $asN \not \subseteq K$ and $Ias  \not \in Ann_R(N)$.
Then there exists $b \in I$ such that $abs N \not = 0$. Now as $N$
is a $S$-2-absorbing second submodule of $M$, $baN \subseteq K$ implies that $bsN \subseteq K$. We show that $IsN \subseteq K$. To see this, let $c$ be an arbitrary element of $I$.
Then $(b + c)aN \subseteq K$. Hence, either $(b + c)sN \subseteq K$ or $(b + c)as  \in Ann_R(N)$. If $(b+c)sN \subseteq K$, then since $bsN \subseteq K$ we have $csN \subseteq K$. If $(b+c)as  \in Ann_R(N)$, then $cas  \not \in Ann_R(N)$, but $ca N \subseteq K$. Thus $csN \subseteq K$. So, we conclude that
$sIN \subseteq K$, as requested.
\end{proof}

\begin{lem}\label{l1.4}
Let $S$ be a m.c.s. of $R$, $I$ and $J$ be two ideals of $R$, let and $N$ be a $S$-2-absorbing second submodule of $M$.
Then there exists a fixed $s \in S$ and whenever $K$ is a submodule of $M$ and $IJN \subseteq K$,
then $sIN \subseteq K$ or $sJN \subseteq K$ or $IJs  \subseteq Ann_R(N)$.
\end{lem}
\begin{proof}
Let $IsN \not \subseteq K$ and $JsN \not \subseteq K$. We show that $IJs  \subseteq Ann_R(N)$. Assume that $c \in I$ and $d\in J$. By assumption there exists $a\in I$ such that
$asN \not \subseteq K$ but $aJN \subseteq K$. Now Lemma \ref{l1.4}
shows that $aJs  \subseteq Ann_R(N)$ and so $(I\setminus (K:_RN))Js \subseteq Ann_R(N)$.
Similarly there exists $b \in (J\setminus (K:_RN))$ such that $Ibs  \subseteq Ann_R(N)$ and also $I(J\setminus (K:_RN))s \subseteq Ann_R(N)$. Thus we have $abs  \in Ann_R(N)$,
$ads  \in Ann_R(N)$ and $cbs  \in Ann_R(N)$. As $a + c \in I$
and $b + d \in J$, we have $(a + c)(b + d)N \subseteq K$.
Therefore, $(a + c)sN \subseteq K$ or $(b + d)sN \subseteq K$ or $(a + c)(b + d) s \in Ann_R(N)$. If $(a + c)sN \subseteq K$, then $csN \not \subseteq K$. Hence $c \in  I\setminus (K:_RN)$
which implies that $cd s \in Ann_R(N)$. Similarly if $(b + d)sN \subseteq K$, we can deduce
that $cds  \in Ann_R(N)$. Finally if $(a+c)(b+d)s  \in Ann_R(N)$, then $(ab+ad+cb+cd)s  \in Ann_R(N)$
so that $cds  \in Ann_R(N)$. Therefore, $IJs  \subseteq Ann_R(N)$.
\end{proof}

Let $M$ be an $R$-module. A proper submodule $N$ of
$M$ is said to be \emph{completely irreducible} if $N=\bigcap _
{i \in I}N_i$, where $ \{ N_i \}_{i \in I}$ is a family of
submodules of $M$, implies that $N=N_i$ for some $i \in I$. It is
easy to see that every submodule of $M$ is an intersection of
completely irreducible submodules of $M$. Thus the intersection
of all completely irreducible submodules of $M$ is zero. \cite{FHo06}.

\begin{rem}\label{r2.1}
Let $N$ and $K$ be two submodules of an $R$-module $M$. To prove $N\subseteq K$, it is enough to show that if $L$ is a completely irreducible submodule of $M$ such that $K\subseteq L$, then $N\subseteq L$  \cite{AF101}.
\end{rem}

Let $S$ be a m.c.s. of $R$ and $M$ be an $R$-module.  $S$-2-absorbing second submodules of $M$ can be characterized in various ways as we
demonstrate in the following theorem.
\begin{thm}\label{t2.3} Let $S$ be a m.c.s. of $R$.  For a submodule $N$ of an $R$-module $M$  with $Ann_R(N) \cap S= \emptyset$ the following statements are equivalent:
\begin{itemize}
 \item [(a)] $N$ is an $S$-2-absorbing second submodule of $M$;
 \item [(b)]  There exists a fixed $s \in S$ such that $s^2abN=s^2aN$ or $s^2abN=s^2bN$ or $s^3abN=0$ for each $a, b \in R$;
 \item [(c)] There exists a fixed $s \in S$ and whenever $abN\subseteq L_1 \cap L_2$, where $a, b \in R$ and $L_1 , L_2$ are
completely irreducible submodules of $M$, implies either that $abs N=0$
 or $saN\subseteq L_1\cap L_2$ or $sbN\subseteq L_1\cap L_2$.
 \item [(d)] There exists a fixed $s \in S$, and $IJN\subseteq K$ implies either that $s IJ \subseteq Ann_R(N)$ or $sIN \subseteq K$ or $sJN \subseteq K$ for each ideals $I, J$ of $R$ and submodule $K$ of $M$.
\end{itemize}
\end{thm}
\begin{proof}
$(b) \Rightarrow (a)$
Let $a, b \in R$ and $K$ be a submodule of $M$ with $abN \subseteq K$. By part (b), there exists a fixed $s \in S$ such that $s^2abN=s^2aN$ or $s^2abN=s^2bN$ or $s^3abN=0$. Thus either $s^3abN=0$ or $s^3aN\subseteq s^2aN=abs^2N \subseteq s^2K\subseteq K$ or  $s^3bN\subseteq s^2bN=abs^2N \subseteq s^2K\subseteq K$. Therefore, by setting $\acute{s}:=s^3$, we have
 either $\acute{s}aN\subseteq K$ or $\acute{s}bN\subseteq K$ or $\acute{s }abN=0$, as needed.

$(a) \Rightarrow (c)$ This is clear.

$(c) \Rightarrow (b)$
By part (c), there exists a fixed $s \in S$. Assume that there are $a, b \in R$ such that
$s^2 abN\not=s^2 aN$ and $s^2abN\not=s^2bN$. Then there exist completely irreducible submodules $L_1$ and $L_2$ of $M$ such that $s^2abN \subseteq L_1$, $s^2abN \subseteq L_2$, $s^2aN \not \subseteq L_1$, and $s^2bN \not \subseteq L_2$ by Remark \ref{r2.1}.
Now $(as)(bs)N=s^2abN \subseteq L_1 \cap L_2$ implies that either $s^2aN \subseteq L_1 \cap L_2$ or $s^2bN \subseteq L_1 \cap L_2$ or $s^3abN=0$ by part (c). If  $s^2aN \subseteq L_1 \cap L_2$ or $s^2bN \subseteq L_1 \cap L_2$, then  $s^2aN \subseteq L_1 $ or $s^2bN \subseteq L_2$, which are contradiction. Thus $s^3abN=0$, as required.

$(a)\Rightarrow (d)$
By Lemma \ref{l1.4}.

$(d)\Rightarrow (a)$
Take $a,b \in R$ and $K$ a submodule of $M$ with $abN \subseteq K$. Now, put $I=Ra$ and $J = Rb$. Then we have $IJN \subseteq K$. By assumption, there is a fixed $s \in S$ such that either $sIJ = s(Ra)(Rb) \subseteq Ann_R(N)$ or $sIN \subseteq K$ or
$sJN \subseteq K$ and so either $sab \in Ann_R(N)$ or $saN \subseteq K$ or $sbN \subseteq K$, as needed.
\end{proof}

\begin{rem}\label{r2.1}
Let $M$ be an $R$-module and $S$ a m.c.s. of $R$. Clearly, every $S$-second submodule of $M$ and every strongly 2-absorbing second submodule $N$ of $M$ with $Ann_R(N) \cap S=\emptyset$ is an $S$-2-absorbing second submodule of $M$. َBut the converse is not true in general, as we can see in the following examples.
\end{rem}

\begin{ex}\label{e24.1}
Consider $\Bbb Z_4$ as an $\Bbb Z$-module. Clearly, $\Bbb Z_4$ is not a strongly 2-absorbing second  $\Bbb Z$-module. Set $S:=\Bbb Z \setminus 2\Bbb Z$. Then for each $s \in S$, $2\Bbb Z_4=2s\Bbb Z_4 \not=s\Bbb Z_4=\Bbb Z_4$ and $2s\Bbb Z_4 \not=0$ implies that  $\Bbb Z_4$ is not an  $S$-second $\Bbb Z$-module. But, if we consider $s=1$, and $n, m \in \Bbb Z$, then
we have three cases:
\begin{description}
\item[Case I] If $n \not =2k$ and $m \not =2k$ for each $k \in \Bbb N$, then
 $$
 nm(1)^2 \Bbb Z_4=\Bbb Z_4=(1)^2n\Bbb Z_4=(1)^2m\Bbb Z_4.
$$
\item[Case II] If $n =2k_1$ and $m =2k_2$ for some $k_1, k_2 \in \Bbb N$, then
 $nm(1)^3\Bbb Z_4=0$.
\item[Case III] If  $n =2k_1$ for some $k_1\in \Bbb N$ and $m \not =2k$ for each $k \in \Bbb N$, then
 $$
 nm(1)^2 \Bbb Z_4=\bar{2}\Bbb Z_4=(1)^2n\Bbb Z_4.
 $$
\end{description}
Thus by Theorem \ref{t2.3} $(b)\Rightarrow (a)$, $\Bbb Z_4$ is an $S$-2-absorbing second $\Bbb Z$-module.
\end{ex}

\begin{ex}
Consider the $\Bbb Z$-module $M=\Bbb Z_{p^\infty}\oplus \Bbb Z_{pq}$, where $p\not=q$ are prime
numbers. Then $M$ is not a strongly 2-absorbing second  $\Bbb Z$-module since
$$
pqM=\Bbb Z_{p^\infty}\oplus 0 \not =\Bbb Z_{p^\infty}\oplus p\Bbb Z_{pq}=pM,
$$
$$
pqM=\Bbb Z_{p^\infty}\oplus 0 \not =\Bbb Z_{p^\infty}\oplus q\Bbb Z_{pq}=qM,
$$ and $pqM=\Bbb Z_{p^\infty}\oplus 0\not =0$. Now, take the m.c.s. $S =\Bbb Z\setminus \{0\}$ and put $s=pq$. Then
$s^2pqM =\Bbb Z_{p^\infty}\oplus 0=s^2pM$ implies that
$M$ is an $S$-2-absorbing second $\Bbb Z$-module by Theorem \ref{t2.3} $(b)\Rightarrow (a)$.
\end{ex}

The following lemma is known, but we write it here for the sake of reference.
\begin{lem}\label{l2.1}
Let $M$ be an $R$-module, $S$ a m.c.s. of $R$,
and $N$ be a finitely generated submodule of $M$. If $S^{-1}N \subseteq  S^{-1}K$ for a submodule
$K$ of $M$, then there exists an $s \in S$ such that $sN \subseteq K$.
\end{lem}
\begin{proof}
This is straightforward.
\end{proof}

Let $S$ be a m.c.s. of $R$. Recall that the saturation $S^*$ of $S$ is defined as $S^*=\{x \in R : x/1\  is \ a\ unit  \ of\ S^{-1}R \}$. It is obvious that $S^*$ is a m.c.s. of $R$ containing $S$ \cite{Gr92}.
\begin{prop}\label{p12.2}
Let $S$ be a m.c.s. of $R$ and $M$ be an R-module. Then we have the following.
\begin{itemize}
\item [(a)] If $N$ is a strongly 2-absorbing second submodule of $M$ such that $S \cap Ann_R(N)=\emptyset$, then $N$ is a $S$-2-absorbing second submodule of $M$. In fact if $S \subseteq u(R)$ and $N$ is an $S$-2-absorbing second submodule of $M$, then $N$ is a strongly 2-absorbing second submodule of $M$.
\item [(b)] If $S_1 \subseteq S_2$ are m.c.s.s of $R$ and $N$ is an $S_1$-2-absorbing second submodule of $M$, then $N$ is an $S_2$-2-absorbing second submodule of $M$ in case $Ann_R(N) \cap S_2=\emptyset$.
\item [(c)] $N$ is an $S$-2-absorbing second submodule of $M$ if and only if $N$ is an  $S^*$-2-absorbing second submodule of $M$
\item [(d)] If $N$ is a finitely generated $S$-2-absorbing second submodule of $M$, then $S^{-1}N$ is a strongly 2-absorbing second submodule of $S^{-1}M$ for some $s \in S$.
\end{itemize}
\end{prop}
\begin{proof}
(a) and (b) These are clear.

(c) Assume that $N$ is an $S$-2-absorbing second submodule of $M$.  We claim that $Ann_R(N) \cap  S^*=\emptyset$. To see this assume that there exists an $x \in Ann_R(N) \cap  S^*$  As $x \in S^*$,  $x/1$ is a unit of $S^{-1}R$ and so $(x/1)(a/s)=1$ for some $a \in R$ and $s \in S$. This yields that $us = uxa$ for some $u \in S$. Now we have that $us = uxa \in  Ann_R(N) \cap  S$,  a
contradiction. Thus, $Ann_R(N) \cap  S^*=\emptyset$. Now as $S\subseteq S^*$, by part (b), $N$ is an  $S^*$-2-absorbing second submodule of $M$. Conversely, assume that $N$ is an $S^*$-2-absorbing second submodule of $M$. Let $rtN \subseteq K$ for some $r, t \in R$. As $N$ is an  $S^*$-2-absorbing second submodule of $M$, there is a fixed $x \in  S^*$ such that $xrt \in Ann_R(N)$ or $xrN \subseteq K$ or $xtN \subseteq K$. As $x/1$ is a unit of $S^{-1}R$, there exist $u, s \in S$ and $a \in R$ such that $us = uxa$. Then
note that $(us)rt = uaxrt \in Ann_R(N)$ or $us(txN) \subseteq K$ or $us(rxN) \subseteq K$. Therefore, $N$ is a $S$-2-absorbing second submodule of $M$.

(d) As $N$ is an $S$-2-absorbing second submodule of $M$, there is a fixed $s \in S$. If $S^{-1}N=0$, then as $N$ is finitely generated, there is an $t \in S$ such that $t \in Ann_R(N)$ by Lemma \ref{l2.1}. Thus $Ann_R(N) \cap  S\not=\emptyset$, a contradiction. So, $S^{-1}N\not=0$. Now let $a/t ,b/h \in S^{-1}R$. As $N$ is an $S$-2-absorbing second submodule of $M$, we have either $abs^2N=as^2N$ or $abs^2N=bs^2N$ or $abs^3N=0$. This implies that either $(a/t)(b/h)S^{-1}N=(a/t)S^{-1}N$ or  $(a/t)(b/h)S^{-1}N=(b/h)S^{-1}N$ or $(a/t)(b/h)S^{-1}N=0$, as needed.
\end{proof}

The following example shows that the converse of Proposition \ref{p12.2} (d) is not true in general.
\begin{ex}\label{e222.2}
Consider the $\Bbb Z$-module $M=\Bbb Q \oplus \Bbb Q\oplus \Bbb Q$, where $\Bbb Q$ is the field of rational numbers. Take the submodule $N = \Bbb Z \oplus \Bbb Z \oplus 0$ and the m.c.s.  $S =\Bbb Z\setminus \{0\}$.
Now, take $s \in S$. Then
there exist prime numbers $p\not= q$ such that $gcd(p, s) = gcd(q, s) = 1$.
Then one can see that $s^2pqN \not=s^2pN$, $s^2pqN \not=s^2qN$, and $s^3pqN \not=0$. Thus $N$ is not an $S$-2-absorbing second submodule of $M$. Since $S^{-1}\Bbb Z =\Bbb Q$
is a field, $S^{-1}(\Bbb Q \oplus \Bbb Q\oplus \Bbb Q)$ is a vector space so that a non-zero submodule $S^{-1}N$ is a strongly 2-absorbing second submodule of $S^{-1}(\Bbb Q \oplus \Bbb Q\oplus \Bbb Q)$.
\end{ex}

\begin{thm}\label{t422.8}
Let $S$ be a m.c.s. of $R$ and $N$ be a submodule of an $R$-module $M$ such that $Ann_R(N) \cap S=\emptyset$. Then $N$ is an $S$-2-absorbing second submodule of $M$ if and only if $s^3N$ is a strongly 2-absorbing second submodule of $M$ for some $s \in S$.
\end{thm}
\begin{proof}
Let $s^3N$ be a strongly 2-absorbing second submodule of $M$ for some $s \in S$ and $a, b \in R$. Then $abs^3N=as^3N$ or $abs^3N=bs^3N$ or $abs^3N=0$ by \cite[Theorem 3.3]{AF16}. Hence $s^6abN=s^6aN$ or $s^6abN=s^6bN$ or $s^9abN=0$. Set $t:=s^3$. Then $t^2abN=t^2aN$ or $t^2abN=t^2bN$ or $t^3abN=0$. Thus by Theorem \ref{t2.3} $(b)\Rightarrow (a)$, $N$ is an $S$-2-absorbing second submodule of $M$. Conversely, suppose that
$N$ is an $S$-2-absorbing second submodule of $M$ and $a, b \in R$. Then for some $s \in S$ we have $s^2abN=s^2aN$ or $s^2abN=s^2bN$ or $s^3abN=0$ by Theorem \ref{t2.3} $(b)\Rightarrow (a)$. This implies that $abs^3N=as^3N$ or $abs^3N=bs^3N$ or $abs^3N=0$.  We not that $Ann_R(N) \cap S=\emptyset$, implies that  $s^3N\not=0$. Therefore,  by
 \cite[Theorem 3.3]{AF16}, $s^3N$ is a strongly 2-absorbing second submodule of $M$.
 \end{proof}

\begin{prop}\label{p22.9}
Let $S$ be a m.c.s. of $R$ and $M$ be an $R$-module. Let $N\subset K$ be two submodules of $M$ and $K$ be a  $S$-2-absorbing second submodule of $M$. Then $K/N$ is a $S$-2-absorbing second submodule of $M/N$.
\end{prop}
\begin{proof}
This is straightforward.
\end{proof}

\begin{prop}\label{p1.6}
Let $S$ be a m.c.s. of $R$ and $N$ be an $S$-2-absorbing second submodule of an $R$-module $M$.
Then we have the following.
\begin{itemize}
  \item [(a)] $Ann_R(N)$ is a $S$-2-absorbing  ideal of $R$.
  \item [(b)] If $K$ is a submodule of $M$ such that $(K:_RN)\cap S=\emptyset$,
   then $(K:_RN)$ is a $S$-2-absorbing ideal of $R$.
  \item [(c)] There exists a fixed $s \in S$ such that $s^nN=s^{n+1}N$, for all $n \geq 3$.
\end{itemize}
\end{prop}
\begin{proof}
(a) Let $a, b, c \in R$ and $abc\in Ann_R(N)$.
Then there exists a fixed $s \in S$ and $abN \subseteq abN$ implies that $asN \subseteq abN$ or $bsN \subseteq abN$ or $sabN=0$. If  $sabN=0$, then we are done. If $asN \subseteq abN$,
then $casN \subseteq cabN=0$. In other case, we do the same.

(b) Let $a, b, c \in R$ and $abc=(ac)b\in (K:_RN)$.
Then  there exists a fixed $s \in S$ such that $acsN \subseteq K$  or $cbsN\subseteq bs N\subseteq K$ or $sabc\in Ann_R(N) \subseteq (K:_RN)$, as needed.

(c) As $N$ is an $S$-2-absorbing second submodule of $M$, there exists a fixed $s \in S$. It is enough to show that $s^3N=s^4N$.
It is clear that $s^4N \subseteq s^3N$.
Since $N$ is an $S$-2-absorbing second submodule, $(s^2)(s^2)N=s^4N \subseteq s^4N$
implies that either $s^3N \subseteq s^4N$ or $s^5N=0$. If $s^5N=0$, then $s^5 \in Ann_R(N) \cap S=\emptyset$ which is a contradiction. Thus $s^3N \subseteq s^4N$, as desired.
\end{proof}

\begin{prop}\label{t121.15}
Let $S$ be a m.c.s. of $R$ and $N$ be a $S$-2-absorbing second submodule of $M$. Then the following statements hold for some $s \in S$.
\begin{itemize}
\item [(a)] $tsN\subseteq thN$ or $hsN\subseteq thN$  for all $t,h \in S$.
\item [(b)] $(Ann_R(N):_Rth)\subseteq (Ann_R(N):_Rts)$ or $(Ann_R(N):_Rth)\subseteq  (Ann_R(N):_Rsh)$ for all $t,h \in S$.
\end{itemize}
\end{prop}
\begin{proof}
(a) Let $N$ be a $S$-2-absorbing second submodule of $M$. Then there is a fixed $s \in S$. Let $L$ be a completely irreducible submodule of $M$ such that $thN \subseteq L$, where $t,h \in S$. Then $tsN \subseteq L$ or $shN \subseteq L$ or $s th \in Ann_R(N)$. As $Ann_R(N)\cap S=\emptyset$, we have  $s th \not\in Ann_R(N)$.
If for each completely irreducible submodule of $M$, we have $tsN \subseteq L$ (resp. $shN \subseteq L$), then we are done by Remark \ref{r2.1}. So suppose that there are completely irreducible submodules $L_1$ and $L_2$ of $M$ such that $tsN \not\subseteq L_1$ and $shN \not\subseteq L_2$. Then since $N$ is a $S$-2-absorbing second submodule of $M$, we conclude that  $hsN \subseteq L_1$ and $stN \subseteq L_2$. Now $htN\subseteq L_1 \cap L_2$ implies that $hsN\subseteq L_1 \cap L_2$ or $stN\subseteq L_1 \cap L_2$. Thus $tsN \subseteq L_1$ or $shN \subseteq L_2$, a contradiction.

(b) This follows from Proposition \ref{t121.15} (a) and Proposition \ref{t191.15} (a).
\end{proof}

An $R$-module $M$ is said to be a \emph{comultiplication module} if for every submodule $N$ of $M$ there exists an ideal $I$ of $R$ such that $N=(0:_MI)$, equivalently, for each submodule $N$ of $M$, we have $N=(0:_MAnn_R(N))$ \cite{AF07}.
\begin{lem}\label{l1.11}
Let $S$ be a m.c.s. of $R$ and $N$ be a submodule of a comultiplication $R$-module $M$. If $Ann_R(N)$ is a $S$-2-absorbing  ideal of $R$, then $N$ is a $S$-2-absorbing second submodule of $M$.
\end{lem}
\begin{proof}
Let $a, b \in R$, $K$ be a submodule of $M$, and $abN\subseteq K$. Then we have $Ann_R(K)abN=0$. As $Ann_R(N)$ is a $S$-2-absorbing  ideal of $R$, there is a fixed $s \in S$ and so $sAnn_R(K)aN=0$ or $sAnn_R(K)bN=0$ or $sabN=0$. If $sabN=0$, we are done. If $sAnn_R(K)aN=0$ or $sAnn_R(K)bN=0$, then $Ann_R(K) \subseteq Ann_R(saN)$ or $Ann_R(K) \subseteq Ann_R(sbN)$. Hence, $saN \subseteq K$ or $sbN \subseteq K$ since $M$ is a comultiplication $R$-module.
\end{proof}

The following example shows that the Lemma \ref{l1.11} is not satisfied in general.

\begin{ex}\label{e1.11}
 By \cite[3.9]{AF07}, the $\Bbb Z$-module $\Bbb Z$ is not a comultiplication $\Bbb Z$-module. Take the m.c.s. $S=\Bbb Z \setminus \{0\}$. The submodule $N=p\Bbb Z$ of $\Bbb Z$, where $p$ is a prime number, is not $S$-2-absorbing second submodule. But $Ann_{\Bbb Z}(p\Bbb Z)=0$ is an $S$-2-absorbing ideal of $\Bbb Z$.
\end{ex}

\begin{prop}\label{c2.10}
Let $S$ be a m.c.s. of $R$ and $M$ be an $R$-module. Then we have the following.
\begin{itemize}
\item [(a)] If $M$ is a multiplication $S$-2-absorbing second $R$-module, then every submodule $N$ of $M$  with $(N:_RM) \cap S=\emptyset$ is a $S$-2-absorbing submodule of $M$.
\item [(b)] If $M$ is a comultiplication $R$-module such that the zero submodule of $M$ is a $S$-2-absorbing submodule, then every submodule $N$ of $M$ with $Ann_R(N) \cap S=\emptyset$ is a $S$-2-absorbing second submodule of $M$.
\end{itemize}
\end{prop}
\begin{proof}
(a) Let $M$ is a multiplication $S$-2-absorbing second $R$-module and $N$ be a submodule of $M$ with $(N:_RM) \cap S=\emptyset$. Then by Proposition \ref{p1.6} (b), $(N:_RM)$ is an $S$-2-absorbing ideal of $R$. Now the result follows from \cite[Proposition 3]{uatk20}.

(b) Let $M$ is a comultiplication $R$-module with the zero submodule of $M$ is a $S$-2-absorbing submodule and $N$ be a submodule of $M$ with $Ann_R(N) \cap S=\emptyset$. We show that $Ann_R(N)$ is an $S$-2-absorbing ideal of $R$. To see this let $a, b, c \in R$ and $abc=(ac)b\in Ann_R(N)$.
Then there exists a fixed $s \in S$ such that $acsN=0$  or $cbsN\subseteq bs N=0$ or $sabc\in Ann_R(M) \subseteq Ann_R(N)$. Thus $Ann_R(N)$ is an $S$-2-absorbing ideal of $R$. Now the result follows from Lemma \ref{l1.11}.
\end{proof}

An $R$-module $M$ satisfies the \emph{double annihilator
conditions} (DAC for short)  if for each ideal $I$ of $R$
we have $I=Ann_R((0:_MI))$ \cite{Fa95}. An
$R$-module $M$ is said to be a \emph{strong comultiplication module} if $M$ is
a comultiplication $R$-module and satisfies the DAC conditions \cite{AF09}.
\begin{thm}\label{t2.5}
Let $M$ be a strong comultiplication $R$-module and $N$ be a submodule of $M$ such that $Ann_R(N) \cap S=\emptyset$, where $S$ is a m.c.s. of $R$. Then the following are equivalent:
\begin{itemize}
\item [(a)] $N$ is an $S$-2-absorbing second submodule of $M$;
\item [(b)]  $Ann_R(N)$ is a $S$-2-absorbing  ideal of $R$;
\item [(c)] $N=(0:_MI)$ for some $S$-2-absorbing  ideal $I$ of $R$ with $Ann_R(N) \subseteq I$.
\end{itemize}
\end{thm}
\begin{proof}
$(a)\Rightarrow (b)$ This follows from Proposition \ref{p1.6} (a).

$(b)\Rightarrow (c)$ As $M$ is a comultiplication $R$-module, $N=(0:_MAnn_R(N))$. Now the result is clear.

$(c)\Rightarrow (a)$
As $M$ satisfies the DAC conditions, $Ann_R((0:_MI))=I$. Now the result follows from Lemma \ref{l1.11}.
\end{proof}

\begin{lem}\label{l02.5}
Let $S$ be a m.c.s. of $R$ and $M$ be an $R$-module. If $N$ is an $S$-second submodule of $M$. Then there exists a fixed $s \in S$ such that
$abN \subseteq K$, where $a, b \in R$ and $K$ is a submodule of $M$ implies that either $sa \in Ann_R(N)$ or $sb \in Ann_R(N)$ or $sN \subseteq K$
\end{lem}
\begin{proof}
Let $N$ be an $S$-second submodule of $M$ and $abN \subseteq K$, where $a, b \in R$ and $K$ is a submodule of $M$. Then $aN \subseteq (K:_Mb)$.
Since $N$ is an $S$-second submodule of $M$, there exists a fixed $s \in S$ such that
$sa \in Ann_R(N)$ or $sbN \subseteq K$. Now, we will show that $sbN \subseteq K$ implies that $sb \in Ann_R(N)$ or $sN \subseteq K$. Assume that $bN \subseteq (K:_Ms)$. Since $N$ is an $S$-second submodule, we get either $sb \in Ann_R(N)$ or $s^2N \subseteq K$.  If $sb \in Ann_R(N)$, then we are done. So assume that $s^2N \subseteq K$. By \cite[Lemma 2.13 (a)]{FF22}, we know that $sN \subseteq s^2N$. Thus we have $sN \subseteq K$.
\end{proof}

\begin{thm}\label{t2.5}
Let $S$ be a  m.c.s. of $R$ and $M$ be an $R$-module. Then the sum of two $S$-second submodules is an $S$-2-absorbing second submodule of $M$.
\end{thm}
\begin{proof}
Let $N_1, N_2$ be two $S$-second submodules of $M$ and $N=N_1+N_2$. Let
$abN \subseteq K$ for some $a, b \in R$ and submodule $K$ of $M$. Since $N_1$ is an $S$-second submodule
submodule of $M$, there exists a fixed $s_1\in S$ such that $s_1a \in Ann_R(N_1)$ or $s_1b \in Ann_R(N_1)$ or $s_1N_1 \subseteq K$ by Lemma \ref{l02.5}. Also, as $N_2$ is an $S$-second submodule of $M$, there exists a fixed
$s_2\in S$ such that $s_2a \in Ann_R(N_2)$ or $s_2b \in Ann_R(N_2)$ or $s_2N_2 \subseteq K$ by Lemma \ref{l02.5}.
Without loss of generality, we may assume that $s_1a \in Ann_R(N_1)$ and $s_2N_2 \subseteq K$. Now,
put $s = s_1s_2 \in S$. This implies that $saN \subseteq K$ and hence $N$ is an S-2-absorbing second
submodule of $M$.
\end{proof}

The following example shows that sum of two $S$-2-absorbing second submodules is not necessarily $S$-2-absorbing second submodule.
\begin{ex}
Consider $M=\Bbb Z_{p^n} \oplus \Bbb Z_{q^n}$ as $\Bbb Z$-module, where $n \in \Bbb N$ and $p, q$ are distinct prime numbers. Set $S =\{x\in \Bbb Z: gcd(x, pq) = 1\}$. Then $S$ is a m.c.s. of $\Bbb Z$. One can see that $\Bbb Z_{p^n}\oplus 0$ and $0 \oplus \Bbb Z_{q^n}$ both are $S$-2-absorbing second submodules. However $p^nM \subseteq 0 \oplus \Bbb Z_{q^n}$, $p^{n-1}xM \not \subseteq 0 \oplus \Bbb Z_{q^n}$,  $pxM \not \subseteq 0 \oplus \Bbb Z_{q^n}$, and $xp^nM \not =0$ implies that $M$ is not an $S$-2-absorbing second $\Bbb Z$-module.
\end{ex}

Let $M$ be an $R$-module. The idealization $R(+)M =\{(a,m): a \in R, m \in  M\}$ of $M$ is
a commutative ring whose addition is componentwise and whose multiplication is defined as $(a,m)(b,\acute{m}) =
(ab, a\acute{m} + bm)$ for each $a, b \in R$, $m, \acute{m}\in M$ \cite{Na62}. If $S$ is a m.c.s. of $R$ and $N$ is a submodule of $M$, then $S(+)N = \{(s, n): s \in S, n \in N\}$ is a m.c.s. of $R(+)M$ \cite{DD02}.

\begin{prop}\label{p2.18}
Let $M$ be an $R$-module and let $I$ be an ideal of $R$ such that $I \subseteq Ann_R(M)$.  Then the following are equivalent:
\begin{itemize}
\item [(a)] $I$ is a strongly 2-absorbing second ideal of $R$;
\item [(b)] $I(+)0$ is a strongly 2-absorbing second ideal of $R(+)M$.
\end{itemize}
\end{prop}
\begin{proof}
This is straightforward.
\end{proof}

\begin{thm}\label{t2.19}
Let $S$ be a m.c.s. of $R$, $M$ be an $R$-module, and $I$ be an ideal of $R$ such that $I \subseteq Ann_R(M)$ and $I \cap S=\emptyset$. Then the following are equivalent:
\begin{itemize}
\item [(a)] $I$ is an $S$-2-absorbing second ideal of $R$;
\item [(b)] $I(+)0$ is an $S(+)0$-2-absorbing second ideal of $R(+)M$;
\item [(c)] $I(+)0$ is an $S(+)M$-2-absorbing second ideal of $R(+)M$.
\end{itemize}
\end{thm}
\begin{proof}
$(a)\Rightarrow (b)$
Let $(a, m ), (b, \acute{m} ) \in R(+)M$. As $I$ is an $S$-2-absorbing second ideal of $R$, there exists a fixed $s \in S$ such that $abs^2I=as^2I$ or $abs^2I=bs^2I$ or $abs^3I=0$.
If  $abs^3I=0$, then $(a,m)(b,\acute{m})(s^3,0)(I(+)0)=0$. If $abs^2I=as^2I$, then we claim that $(a,m)(b,\acute{m})(s^2,0)(I(+)0)=(a,m)(s^2,0)(I(+)0)$. To see this let
$(s^2x,0)(a,m)=(s^2,0)(x,0)(a,m) \in (s^2,0)(a,m)(I(+)0)$. As $abs^2I=as^2I$, we have $s^2ax=abs^2y$ for some $y \in I$. Thus as $y \in I \subseteq Ann_R(M)$,

$$
(s^2,0)(x,0)(a,m)=(s^2xa,0)=(abs^2y,0)=(s^2,0)(y,0)(a,m)(b,\acute{m})
$$
Hence,
$(s^2,0)(x,0)(a,0) \in (a,m)(b,\acute{m})(s^2,0)(I(+)0)$ and so $(a,m)(b,\acute{m})(s^2,0)(I(+)0)=(a,m)(s^2,0)(I(+)0)$.
Since the inverse inclusion is clear we reach the claim.

$(b)\Rightarrow (c)$
Since $S(+)0 \subseteq S(+)M$, the result follows from Proposition \ref{p12.2} (b).

$(c)\Rightarrow (a)$
Let $a,b \in R$. As $I(+)0$ is an $S(+)M$-2-absorbing second ideal of $R(+)M$, there exists a fixed $(s, m) \in S(+)M$ such that
$$
(a,0)(b,0)(s,m)^2(I(+)0)=(a,0)(s,m)^2(I(+)0)
$$
or
$$
(a,0)(b,0)(s,m)^2(I(+)0)=(b,0)(s,m)^2(I(+)0)
$$
or
$$
(a,0)(b,0)(s,m)^3(I(+)0)=0.
$$
If $(ab,0)(s,m)^2(I(+)0)=0$, then for each $abs^2x \in abs^2I$ we have
$$
0=(ab,0)(s,m)^2(x,0)=(ab,0)(s^2,2sm)(x,0)=(abs^2,2absm)(x,0)
$$
$$
=(abs^2,0)(x,0)=(abs^2x,0).
$$
Thus $abs^2I=0$. If
$(ab,0)(s,m)^2(I(+)0)=(a,0)(s,m)^2(I(+)0)$, then we claim that $abs^2I=as^2I$. To see this, let
$s^2xa\in s^2aI$. Then for some $y \in I$, as $x \in I\subseteq Ann_R(M)$ we have
$$
(s^2ax,0)=(s^2ax,2sxm)=(s,m)^2(a,0)(x,0)=(s,m)^2(aby,0)
$$
$$
=(s^2aby,2sabmy)=(s^2aby,0).
$$
Hence,
$s^2ax \in abs^2I$ and so $s^2aI \subseteq s^2abI$ Thus  $s^2aI = s^2abI$. Similarly, if $(ab,0)(s,m)^2(I(+)0)=(b,0)(s,m)^2(I(+)0)$, then $s^2bI \subseteq s^2abI$, and so $s^2bI=s^2abI$.
\end{proof}

Let $R_i$ be a commutative ring with identity, $M_i$ be an $R_i$-module for each $i = 1, 2,..., n$, and $n \in \Bbb N$. Assume that
$M = M_1\times M_2\times ...\times M_n$ and $R = R_1\times R_2\times ...\times R_n$. Then $M$ is clearly
an $R$-module with componentwise addition and scalar multiplication. Also,
if $S_i$ is a multiplicatively closed subset of $R_i$ for each $i = 1, 2,...,n$,  then
$S = S_1\times S_2\times ...\times S_n$ is a multiplicatively closed subset of $R$. Furthermore,
each submodule $N$ of $M$ is of the form $N = N_1\times N_2\times...\times N_n$, where $N_i$ is a
submodule of $M_i$ for each $i = 1, 2,..., n$.
\begin{thm}\label{t11.15}
Let $R = R_1 \times R_2$ and $S = S_1\times S_2$ be a  m.c.s. of
$R$, where $R_i$ is a commutative ring
with $1\not= 0$ and $S_i$ is a  m.c.s. of $R_i$ for each $i = 1, 2$. Let $M = M_1 \times M_2$
be an $R$-module, where $M_1$ is an $R_1$-module and $M_2$ is an $R_2$-module. Suppose that $N = N_1 \times N_2$ is a  submodule of $M$. Then the following conditions are equivalent:
\begin{itemize}
  \item [(a)] $N$ is an $S$-2-absorbing second submodule of $M$;
  \item [(b)] Either  $Ann_{R_1}(N_1) \cap S_1\not=\emptyset$ and $N_2$ is a $S_2$-2-absorbing second submodule of $M_2$ or  $Ann_{R_2}(N_2) \cap S_2\not=\emptyset$ and $N_1$ is a $S_1$-2-absorbing second submodule of $M_1$ or $N_1$ is an $S_1$-second submodule of $M_1$ and $N_2$ is an $S_2$-second submodule of $M_2$ .
\end{itemize}
\end{thm}
\begin{proof}
$(a)\Rightarrow (b)$
Let $N = N_1 \times N_2$ be a $S$-2-absorbing second submodule of $M$. Then $Ann_R(N)= Ann_{R_1}(N_1)\times Ann_{R_2}(N_2)$ is an $S$-2-absorbing ideal of $R$ by Proposition \ref{p1.6} (a). Thus, either $Ann_R(N_1) \cap S_1 = \emptyset$ or  $Ann_R(N_2) \cap S_2 = \emptyset$. Assume that $Ann_R(N_1) \cap S_1\not = \emptyset$. We show that
$N_2$ is an $S_2$-2-absorbing second submodule of $M_2$. To see this, let $t_2r_2N_2 \subseteq K_2$ for some $t_2,r_2 \in R_2$ and a submodule $K_2$ of $M_2$. Then $(1,t_2)(1,r_2)(N_1 \times N_2) \subseteq M_1 \times K_2$. As $N$ is an $S$-2-absorbing second submodule of $M$, there exists a fixed $(s_1,s_2) \in S$ such that $(s_1,s_2)(1,r_2)(N_1 \times N_2) \subseteq M_1 \times K_2$ or $(s_1,s_2)(1,t_2)(N_1 \times N_2) \subseteq M_1 \times K_2$ or $(s_1,s_2)(1,t_2)(1,r_2)(N_1 \times N_2)=0$. It follows that either $s_2r_2N_2 \subseteq K_2$ or $s_2t_2N_2 \subseteq K_2$ or $s_2t_2r_2N_2=0$ and so $N_2$ is an $S_2$-2-absorbing second submodule of $M_2$. Similarly, if $Ann_{R_2}(N_2) \cap S_2\not=\emptyset$, then one can see that $N_1$ is an $S_1$-2-absorbing second submodule of $M_1$. Now assume that $Ann_R(N_1) \cap S_1 = \emptyset$ and $Ann_R(N_2) \cap S_2 = \emptyset$. We will show that $N_1$ is an $S_1$-second submodule of $M_1$ and $N_2$ is an $S_2$-second submodule of $M_2$.
First, note that there exists a fixed $s = (s_1, s_2) \in S$ satisfying $N$ to be an $S$-2-absorbing second submodule
of $M$. Suppose that $N_1$ is not an $S_1$-second submodule of $M_1$. Then
there exists $a \in R_1$ and a submodule $K_1	$ of $M_1$ such that $aN_1 \subseteq K_1$ but $s_1a \not \in Ann_R(N_1)$ and $s_1N_1 \not \subseteq K_1$. On the other hand $Ann_R(N_2) \cap S_2 = \emptyset$ and $s_2 \not\in Ann_R(N_2)$ so that $s_2N_2\not=0$. Thus by Remark \ref{r2.1}, there exists a completely irreducible submodule $L_2$ of $M_2$ such that $s_2N_2 \not \subseteq L_2$. Also note that
$$
(a, 1)(1, 0)N=(a, 1)(1, 0)(N_1 \times N_2)=aN_1 \times 0\subseteq K_1 \times 0\subseteq  K_1 \times L_2.
$$
Since $N$ is an $S$-2-absorbing second submodule of $M$, we have either
$(s_1, s_2)(1, 0)N\subseteq K_1 \times L_2$ or $(s_1, s_2)(a, 1)N\subseteq K_1 \times L_2$ or $(s_1, s_2)(a, 1)(1, 0)N=0$.
Hence, we conclude that either $s_1N_1\subseteq K_1$ or $s_2aN_2\subseteq L_2$ or $s_1aN_1=0$,
which them are contradictions. Thus, $N_1$ is
an $S_1$-second submodule of $M_1$.  Similar argument shows that $N_2$ is an $S_2$-second submodule of $M_2$.

$(b)\Rightarrow (a)$
Assume that $N_1$ is an $S_1$-2-absorbing second submodule of $M_1$ and $Ann_{R_2}(N_2) \cap S_2\not=\emptyset$.
we will show that $N$ is an $S$-2-absorbing second submodule of $M$.
Then there exists an $s_2 \in Ann_{R_2}(N_2) \cap S_2$.  Let $(r_1, r_2)(t_1,t_2)(N_1 \times N_2) \subseteq K_1 \times K_2$ for some $t_i, r_i \in R_i$ and submodule $K_i$ of $M_i$, where $i = 1, 2$. Then $r_1t_1N_1 \subseteq K_1$.  As $N_1$ is an
$S_1$-2-absorbing second submodule of $M_1$,  there exists a fixed $s_1 \in S_1$ such that $s_1r_1N_1 \subseteq K_1$ or  $s_1t_1N_1 \subseteq K_1$ or $s_1r_1t_1N_1 =0$.
Now we set $s =(s_1, s_2)$. Then $s(r_1, r_2)(N_1 \times N_2) \subseteq K_1 \times K_2$ or $s(t_1, t_2)(N_1 \times N_2) \subseteq K_1 \times K_2$ or $s(r_1, r_2)(t_1, t_2)(N_1 \times N_2)=0$.
Therefore, $N$ is an $S$-2-absorbing second submodule of $M$. Similarly one can show that if $N_2$ is an $S_2$-2-absorbing second submodule of $M_2$ and $Ann_{R_1}(N_1) \cap S_1\not=\emptyset$, then $N$ is an $S$-2-absorbing second submodule of $M$.
Now assume that $N_1$ is an $S_1$-second submodule of $M_1$ and $N_2$ is an $S_2$-second submodule of $M_2$ . Let $a, b \in R_1$, $x, y \in R_2$, $K_1$ is a submodule of $M_1$ and $K_2$ is a submodule of $M_2$ such that
$$
(a, x)(b, y)N=(a, x)(b, y)(N_1 \times N_2 \subseteq K_1 \times K_2.
$$
Then we have $abN_1 \subseteq K_1$ and $xyN_2 \subseteq K_2$.  Since $N_1$ is an $S_1$-second submodule of $M_1$, there exists a fixed $s_1 \in S_1$ such that either $s_1a \in Ann_{R_1}(N_1)$ or $s_1b \in Ann_{R_1}(N_1)$
or $s_1N_1 \subseteq K_1$ by Lemma \ref{l02.5}. Similarly, there exists a fixed $s_2 \in S_2$ such that either $s_2x \in Ann_{R_2}(N_2)$ or $s_2y \in Ann_{R_2}(N_2)$
or $s_2N_2 \subseteq K_2$ by Lemma \ref{l02.5}. Also without loss of generality, we may assume that $s_1a \in Ann_{R_1}(N_1)$ and
$s_2N_2 \subseteq K_2$ or $s_1a \in Ann_{R_1}(N_1)$ and $s_2x \in Ann_{R_2}(N_2)$ or $s_1N_1 \subseteq K_1$ and  $s_2N_2 \subseteq K_2$. If $s_1a \in Ann_{R_1}(N_1)$ and
$s_2N_2 \subseteq K_2$, then we have
$$
(s_1,s_2)(a, x)(N_1\times N_2 )=s_1aN_1\times s_2xN_2 \subseteq 0 \times K_2\subseteq K_1 \times K_2.
$$
If $s_1a \in Ann_{R_1}(N_1)$ and $s_2x \in Ann_{R_2}(N_2)$, then $(s_1,s_2)(a, x)(b,y)(N_1\times N_2 )=0$.
If $s_1N_1 \subseteq K_1$ and  $s_2N_2 \subseteq K_2$, then
$$
(s_1,s_2)(a, x)(b,y)(N_1\times N_2 )\subseteq (s_1,s_2)N\subseteq K_1 \times K_2.
$$
Hence, $N$ is an $S$-2-absorbing second submodule of $M$.
\end{proof}

The following example shows that if $N_1$ is an $S_1$-2-absorbing second
submodule of $M_1$ and $N_2$ is an $S_2$-2-absorbing second submodule of $M_2$, then $N_1\times N_2$
may not be an $S_1\times S_2$-2-absorbing second submodule of $M_1 \times M_2$ in general.
\begin{ex}
Consider the $\Bbb Z$-modules $M_1 = \Bbb Z_9$ and $M_2 =\Bbb Z_4$. Let $S_1 = \Bbb Z \setminus 3\Bbb Z$ and $S_2 = \Bbb Z \setminus 2\Bbb Z$. Then $M_1$ and $M_2$ are $S_1$ and $S_2$-2-absorbing second modules (see Example \ref{e24.1}). But $M=M_1\times M_2$ is not an
$S=S_1\times S_2$-2-absorbing second module since $(1, 2)(3, 1) M\subseteq  \bar{3}\Bbb Z_9\times \bar{2}\Bbb Z_4$
but for each $s = (s_1, s_2) \in S$, $s(3, 1) M\not\subseteq   \bar{3}\Bbb Z_9\times \bar{2}\Bbb Z_4$, $s(1, 2) N\not\subseteq   \bar{3}\Bbb Z_9\times \bar{2}\Bbb Z_4$, and $s(1, 2)(3, 1)M\not =0$.
\end{ex}

\begin{thm}\label{t2.7}
Let  $M = M_1 \times M_2\times ... \times M_n$ be an $R = R_1 \times R_2\times ...\times R_n$-module and $S = S_1\times S_2\times ... \times S_n$ be a  m.c.s. of
$R$, where $M_i$ is an $R_i$-module and $S_i$ is a m.c.s. of $R_i$ for each $i = 1, 2,...,n$. Let $N=N_1\times N_2\times ... \times N_n$ be a submodule of $M$. Then the following are equivalent:
\begin{itemize}
\item [(a)] $N$ is an $S$-2-absorbing second submodule of $M$;
\item [(b)] $N_k$ is an $S_k$-2-absorbing second submodule of $M_k$ for some $k \in \{1, 2, . . . , n\}$ and
$Ann_{R_t}(N_t)\cap S_t\not = \emptyset$ for each  $t \in \{1, 2, . . . , n\} \setminus \{k\}$ or $N_{k_1}$ is an $S_{k_1}$-second submodule of $M_{k_1}$ and $N_{k_2}$ is an $S_{k_2}$-second submodule of $M_{k_2}$ for some $k_1,k_2 \in \{1, 2, . . . , n\}$ ($k_1 \not=k_2$) and $Ann_{R_t}(N_t)\cap S_t\not = \emptyset$ for each  $t \in \{1, 2, . . . , n\} \setminus \{k_1,k_2\}$.
\end{itemize}
\end{thm}
\begin{proof}
We apply induction on $n$. For $n = 1$, the result is true. If $n = 2$, then the result follows from Theorem \ref{t11.15}. Now assume that parts (a) and (b) are equal when $k < n$. We shall prove $(b)\Leftrightarrow (a)$ when $k = n$.
Put $R =\acute{ R} \times R_n$, $M =\acute{ M} \times M_n$, and $S =\acute{ S} \times S_n$, where $\acute{R}=R_1\times R_2\times ... \times R_{n-1}$, $\acute{M}=M_1\times M_2\times ... \times M_{n-1}$, and $\acute{S}=S_1\times S_2\times ... \times S_{n-1}$.  Also, $N =\acute{ N} \times N_n$, where $\acute{N}=N_1\times N_2\times ... \times N_{n-1}$. Then by Theorem \ref{t11.15}, $N$ is an $S$-2-absorbing second submodule of $M$ if
and only if $Ann_{\acute{ R}}(\acute{ N})\cap \acute{ S}\not = \emptyset$ and $N_n$ is an $S_n$-2-absorbing second submodule of
$M_n$ or $\acute{N}$ is an $\acute{S}$-2-absorbing second submodule of
$\acute{M}$  and $Ann_{R_n}(N_n)\cap S_n\not = \emptyset$
or $\acute{N}$ is an $\acute{S}$-second submodule of
$\acute{M}$  and $N_n$ is an $S_n$-second submodule of
$M_n$.  Now the rest follows from induction hypothesis and \cite[Theorem 2.12]{FF22}.
\end{proof}

For a submodule $N$ of an $R$-module $M$ the \emph{second radical} (or \emph{\emph{second socle}}) of $N$ is defined  as the sum of all second submodules of $M$ contained in $N$ and it is denoted by $sec(N)$ (or $soc(N)$). In case $N$ does not contain any second submodule, the second radical of $N$ is defined to be $(0)$ (see \cite{CAS13} and \cite{AF11}).

\begin{thm}\label{p111.11}
Let $M$ be a finitely generated comultiplication $R$-module. If $N$ is a $S$-2-absorbing second submodule of $M$, then $sec(N)$ s a  $S$-2-absorbing second submodule of $M$.
\end{thm}
\begin{proof}
Let  $N$ be a $S$-2-absorbing second submodule of $M$. By Proposition \ref{p1.6} (a), $Ann_R(N)$ is an $S$-2-absorbing ideal of $R$. Thus by Lemma \ref{p9.11}, $\sqrt{Ann_R(N)}$ is an $S$-2-absorbing ideal of $R$. By \cite[2.12]{AF25}, $Ann_R(sec(N))=\sqrt{Ann_R(N)}$.  Therefore, $Ann_R(sec(N))$ is an $S$-2-absorbing ideal of $R$. Now the result follows from  Lemma \ref{l1.11}.
\end{proof}

\begin{prop}\label{t92.16}
Let $S$ be a m.c.s. of $R$ and $f : M \rightarrow \acute{M}$ be a monomorphism of R-modules. Then we have the following.
\begin{itemize}
  \item [(a)] If $N$ is an $S$-2-absorbing second submodule of $M$, then $f(N)$ is an $S$-2-absorbing second submodule of $\acute{M}$.
  \item [(b)] If $\acute{N}$ is an $S$-2-absorbing second submodule of $\acute{M}$ and $\acute{N} \subseteq f(M)$, then $f^{-1}(\acute{N})$ is an $S$-2-absorbing second submodule of $M$.
 \end{itemize}
\end{prop}
\begin{proof}
(a) As $Ann_R(N) \cap S=\emptyset$ and $f$ is a monomorphism, we have $Ann_R(f(N)) \cap S=\emptyset$. Let $a, b \in R$. Since $N$ is an $S$-2-absorbing second submodule of $M$, there exists a fixed $s \in S$ such that $s^2abN=s^2aN$ or $s^2abN=s^2bN$ or $s^3abN=0$. Thus
 $s^2abf(N)=s^2af(N)$ or $s^2abf(N)=s^2bf(N)$ or $s^3abf(N)=0$, as needed.

(b)  $Ann_R(\acute{N}) \cap S=\emptyset$ implies that $Ann_R(f^{-1}(\acute{N})) \cap S=\emptyset$.
Now let $a, b \in R$. As $\acute{N}$ is an $S$-2-absorbing second submodule of $M$, there exists a fixed $s \in S$ such that $s^2ab\acute{N}=s^2a\acute{N}$ or $s^2ab\acute{N}=s^2b\acute{N}$ or $s^2ab\acute{N}=0$. Therefore,
$s^2abf^{-1}(\acute{N})=s^2af^{-1}(\acute{N})$ or $s^2abf^{-1}(\acute{N})=s^2bf^{-1}(\acute{N})$ or $s^2abf^{-1}(\acute{N})=0$, as requested.
\end{proof}

\begin{thm}\label{l8.11} Let $S$ be a m.c.s. of $R$ and let $M$ be an $R$-module. If $E$ is an injective $R$-module and $N$ is an $S$-2-absorbing submodule of $M$ such that $Ann_R(Hom_R(M/N,E)) \cap S \not =\emptyset$, then $Hom_R(M/N,E)$ is a $S$-2-absorbing second $R$-module.
\end{thm}
\begin{proof}
Let $a, b \in R$. Since $N$ is an $S$-2-absorbing submodule
of $M$, there is a fixed $s \in S$ such that either $(N:_Mabs^2)=(N:_Mas^2)$ or $(N:_Mabs^2)=(N:_Mbs^2)$ or $(N:_Mabs^3)=M$ by Theorem \ref{l181.4}.
Since $E$ is an injective $R$-module, by replacing $M$ with $M/N$ in \cite[Theorem 3.13 (a)]{AF101}, we have $Hom_R(M/(N:_Ma), E)=aHom_R(M/N,E)$. Therefore,
$$
abs^2Hom_R(M/N, E)=Hom_R(M/(N:_Mabs^2), E)=
$$
$$
Hom_R(M/(N:_Mas^2), E)=as^2Hom_R(M/N,E)
$$
or
$$
abs^2Hom_R(M/N, E)=Hom_R(M/(N:_Mabs^2), E)=
$$
$$
Hom_R(M/(N:_Mbs^2), E)=bs^2Hom_R(M/N,E)
$$
or
$$
abs^3Hom_R(M/N, E)=Hom_R(M/(N:_Mabs^3), E)=
$$
$$
Hom_R(M/M, E)=0,
$$
as needed
\end{proof}

\begin{thm}\label{t8.13}
Let $M$ be a $S$-2-absorbing second $R$-module and $F$ be a right exact linear covariant functor over the category of $R$-modules. Then $F(M)$ is a $S$-2-absorbing second $R$-module if $Ann_R(F(M) ) \cap S\not =\emptyset$.
\end{thm}
\begin{proof}
 This follows from \cite[Lemma 3.14]{AF101} and Theorem \ref{t2.3} $(a) \Leftrightarrow (b)$.
\end{proof}

\bibliographystyle{amsplain}

\end{document}